\documentclass[a4paper,11pt]{amsart}

\usepackage{amstext,amsfonts,amsthm,graphicx,amssymb,amscd,epsfig}
\usepackage{amsmath}
\usepackage[ansinew]{inputenc}    

\usepackage{mathrsfs}

\theoremstyle{definition}
\newtheorem{definition}{Definition}[section]
\newtheorem{ex}[definition]{Example}
\newtheorem{rem}[definition]{Remark}

\theoremstyle{plain}

\newtheorem{lem}[definition]{Lemma}

\newtheorem{teo}[definition]{Theorem}

\newfont{\bbb}{msbm10 scaled\magstephalf}     
\def\C{\mathbb C}
\def\K{\mathbb K}
\def\R{\mathbb R}

\def\R{\mbox{\bbb R}}

\def\O{\mathcal O}
\def\A{\mathscr A}
\def\Lift{\operatorname{Lift}}
\def\Derlog{\operatorname{Derlog}}
\def\Aecod{\operatorname{\mbox{$\A_e$-cod}}}

\title{Liftable vector fields, unfoldings and augmentations}

\author{J. J. Nuño-Ballesteros, R. Oset Sinha}

\date{}

\address{Departament de Matemàtiques,
Universitat de Val\`encia, Campus de Burjassot, 46100 Burjassot,
Spain}

\email{Juan.Nuno@uv.es}

\email{Raul.Oset@uv.es}

\thanks{Work partially supported by DGICYT Grant MTM2015--64013--P}
\subjclass[2000]{Primary 58K40; Secondary 58K20, 32S05} \keywords{Liftable vector fields, augmentations, stable unfoldings}

\begin{document}
\begin{abstract}
We study liftable vector fields of smooth map-germs. We show how to obtain the module of liftable vector fields of any map-germ of finite singularity type from the module of liftable vector fields of a stable unfolding of it. As an application, we obtain the liftable vector fields for the family $H_k$ in Mond's list. We then show the relation between the liftable vector fields of a stable germ and its augmentations.
\end{abstract}

\maketitle

\section{Introduction}

Since the beginning of the study of singularities of differentiable maps with Whitney and Thom, classification of singularities has been one of the main research subjects. Classical techniques, even with the help of computers, have, in a way, exhausted their potential and forced the appearance of new methods in order to attend the growing need of harder classifications. Examples of these new methods can be found in \cite{robertamond,ORW1}. In \cite{arnold} Arnol'd defined liftable vector fields and showed how they can be used for classification purposes. In fact, the new methods developed rely in many aspects in the computation of liftable vector fields. However, it is not easy at all to compute liftable vector fields in general. 

Here we consider $\A$-classification of map-germs $f:(\K^n,S)\to (\K^p,0)$, where $\K=\R$ or $\K=\C$, $S\subset\K$ is a finite subset and the map-germ is either smooth when $\K=\R$ or holomorphic when $\K=\C$. In \cite{NORW}, the authors gave a method to construct liftable vector fields for a certain class of map-germs $f$ which includes stable germs. They then showed how to obtain liftable vector fields of a finite singularity type germ $f$ from liftable vector fields of $F$, a stable unfolding of it. By finite singularity type, we mean a germ which admits a stable unfolding with a finite number of parameters. Furthermore, if $F$ is a one-parameter stable unfolding, they showed that any liftable vector field of $f$ comes from a liftable vector fields of $F$ by the previous construction. Although the class of map-germs which admit one-parameter stable unfoldings seems to be quite big amongst simple germs, it is still not fully understood when this is the case.

In this paper we generalize the result in \cite{NORW}, that is, we prove that all liftable vector fields of any finite singularity type map-germ can be obtained from the liftable vector fields of a stable unfolding of it, with no restriction on the number of parameters. The main result is that if $F$ is an $r$-parameter stable unfolding of $f$, then
$$
\Lift(f)=\pi_1(i^*(\Lift(F)\cap G)),
$$
where $\pi_1$ is the projection onto the first $p$ components, $i^*$ is the morphism induced by $i(X)=(X,0)$ and $G$ is the submodule of $\theta_{p+r}$ generated by $\partial/\partial X_k$, with $1\leq k\leq p$ and $\Lambda_i\partial/\partial \Lambda_j,1\leq i, j\leq r$ (here we denote by $(X,\Lambda)$ the coordinates in $\C^p\times\C^r$). As an application, we obtain the liftable vector fields for the family $H_k$ in Mond's list of simple germs $f:(\K^2,0)\to (\K^3,0)$ \cite{mond}. 

In the last section, we turn our attention to families of singularities which can be obtained by the method of augmentation. We show the relation between the liftable vector fields of a one-parameter stable unfolding of a map-germ and of its augmentations. We then use this result to prove the following:  let $F$ be a one-parameter stable unfolding of $f:(\C^n,S)\to (\C^p,0)$ with $p\le n+1$ and let $Af$ be the augmentation of $f$ by $F$ and a quasihomogeneous function $\phi$, then
$$
\pi_2(i^*(\Lift(Af)))=\pi_2(i^*(\Lift(F))),
$$
where now $\pi_2$ is the projection onto the last component. This property
allows us to use some results in \cite{casonattooset,ORW1,ORW2}, where this condition is required.


\section{Notation}\label{section-notation}

We consider map-germs $f:(\K^n,S)\to (\K^p,0)$, where $\K=\R$ or $\K=\C$, and $S\subset \K^n$ a finite subset. For simplicity, we will say that $f$ is smooth if it is smooth (i.e. $C^\infty$) when $\K=\R$ or holomorphic when $\K=\C$. We denote by $\O_n=\O_{\K^n,S}$ and $\O_p=\O_{\K^p,0}$ the rings of smooth function germs in the source and target respectively and by $\theta_n=\theta_{\K^n,S}$ and $\theta_p=\theta_{\K^p,0}$ the corresponding modules of vector field germs. The module of vector fields along $f$ will be denoted by $\theta(f)$. Associated with $\theta(f)$ we have two morphisms $tf:\theta_n\rightarrow \theta(f)$, given by $tf(\chi)=df\circ\chi$,
and $wf:\theta_p\rightarrow \theta(f)$, given by $wf(\eta)=\eta\circ f$. 
The $\A_e$-tangent space and the $\A_e$-codimension of $f$ are defined respectively as 
$$T\A_e f=tf(\theta_n)+wf(\theta_p),\quad 
\Aecod(f)=\dim_\K\frac{\theta(f)}{T\A_e f}.$$
It follows from Mather's infinitesimal stability criterion \cite{MaII} that a germ is stable if and only if its $\A_e$-codimension is 0. We refer to Wall's survey paper \cite{Wall} for general background on the theory of
singularities of mappings.

We will say that $f$ has finite singularity type if 
$$\dim_\K\frac{\theta(f)}{tf(\theta_n)+(f^*m_p)\theta_n}<\infty,$$
where $m_p$ is the maximal ideal of $\O_p$ and $f^*:\O_p\to\O_n$ is the induced map. Another remarkable result of Mather is that $f$ has finite singularity type if and only if it admits an $r$-parameter stable unfolding (see \cite{Wall}). We recall that an $r$-parameter unfolding of $f$ is another map-germ 
$$f:(\K^n\times\K^r,S\times\{0\})\to (\K^p\times\K^r,0)
$$ 
of the form $F(x,\lambda)=(f_\lambda(x),\lambda)$ and such that $f_0=f$. 

Along the paper, we use the notation of small letters $x_1,\dots,x_n,\lambda_1,\dots,\lambda_r$ for the coordinates in $\K^n\times\K^r$ and capital letters $X_1,\dots,X_p,\Lambda_1,\dots,\Lambda_r$ for the coordinates in $\K^p\times\K^r$.

\begin{definition} Let $f:(\K^n,S)\to (\K^p,0)$ be a smooth map-germ.
\begin{enumerate}
\item A vector field germ $\eta\in \theta_p$ is called \textit{liftable
over $f$}, if there exists $\xi\in\theta_n$ such that
$df\circ\xi=\eta\circ f$ (i.e., $tf(\xi)=wf(\eta)$). The set of vector
field germs liftable over $f$ is denoted by $\Lift(f)$ and is an
$\mathcal{O}_p$-submodule of $\theta_p$.

\item We also define $\widetilde{\tau}(f)=ev_0(\Lift(f))$, the subspace of $T_0\K^p$ given by the evaluation at
the origin of elements of $\Lift(f)$.
\end{enumerate}
\end{definition}

When $f$ has finite $\A_e$-codimension, the space $\widetilde{\tau}(f)$ can be interpreted geometrically as the tangent space to the isosingular locus (or analytic stratum in Mather's terminology) of $f$. This is the well
defined manifold in the target along which the map $f$ is trivial. See \cite{robertamond,houston} for some basic properties of $\widetilde\tau(f)$.

When $\K=\C$ and $f$ has finite singularity type, we always have the inclusion $\Lift(f)\subseteq \Derlog(\Delta(f))$,  where $\Delta(f)$ is the discriminant of $f$ (i.e., the image of non submersive points of $f$) and $\Derlog(\Delta(f))$ is the submodule of $\theta_p$ of vector fields which are tangent to $\Delta(f)$. Moreover, we have the equality $\Lift(f)=\Derlog(\Delta(f))$ in case $f$ has finite $\A_e$-codimension (see \cite{damon,nunomond}).

\begin{definition}
Let $h:(\mathbb{K}^n,S)\rightarrow (\mathbb K^p,0)$ be a map-germ
with a 1-parameter stable unfolding $H(x,\lambda)=(h_{\lambda}(x),\lambda)$. Let $g:(\mathbb K^q,0)\rightarrow (\mathbb K,0)$ be a function-germ. Then, the \textit{augmentation of h by H and g}
is the map $A_{H,g}(h)$ given by $(x,z)\mapsto (h_{g(z)}(x),z)$.
\end{definition}

\section{Liftable vector fields and stable unfoldings}

In this section we show that all the liftable vector fields of any map-germ of finite singularity type can be obtained from the liftable vector fields of a stable unfolding (not necessarily versal) of it.

\begin{teo}\label{lift}
Let $f:(\mathbb K^n,S)\longrightarrow (\mathbb K^p,0)$ be a
non-stable germ of finite singularity type and let $F$ be an $r$-parameter stable unfolding. Then,
$$
\Lift(f)=\pi_1(i^*((\Lift(F)\cap G)),
$$
where $\pi_1$ is the projection onto the first $p$ components, $i^*$ is the morphism induced by $i(X)=(X,0)$ and $G$ is the submodule of $\theta_{p+r}$ generated by $\partial/\partial X_k$, with $1\leq k\leq p$ and $\Lambda_i\partial/\partial \Lambda_j,1\leq i, j\leq r$.
\end{teo}
\begin{proof}
The proof is similar to the proof in Theorem 2 in \cite{NORW}. We will prove it by a double inclusion. For simplicity, we set $\varphi=\pi_1\circ i^*:\theta_{p+r}\to\theta_p$.

Consider $\eta\in \Lift(F)\cap G$. There exists $\xi\in \theta_{n+r}$ such
that $dF\circ \xi=\eta\circ F$. Evaluating this system of equations in $\lambda=0$ we have the following:
\[
\left( \frac{\partial f_{\lambda}}{\partial x}(x, 0) \right) \left(
\begin{array}{c}
\xi_1(x, 0) \\
\vdots \\
\xi_n(x, 0)
\end{array}
\right) + \left( \frac{\partial f_{\lambda}}{\partial \lambda}(x, 0)
\right) \left(
\begin{array}{c}
\xi_{n+1}(x, 0) \\
\vdots \\
\xi_{n+r}(x, 0)
\end{array}
\right) = \left(
\begin{array}{c}
\eta_1(F(x, 0)) \\
\vdots \\
\eta_p(F(x, 0))
\end{array}
\right)
\]
and
\[
\xi_{n+k}(x,0)=\eta_{p+k}(F(x,0))=\eta_{p+k}(f(x),0)\quad (1\le
k\le r).
\]
Since $\eta\in G$, $\xi_{n+k}(x,0)=\eta_{p+k}(f(x),0)=0$, and so
\[
\left( \frac{\partial f_{\lambda}}{\partial x}(x, 0) \right) \left(
\begin{array}{c}
\xi_1(x, 0) \\
\vdots \\
\xi_n(x, 0)
\end{array}
\right) = \left(
\begin{array}{c}
\eta_1(F(x, 0)) \\
\vdots \\
\eta_p(F(x, 0))
\end{array}
\right),
\]
which means that $\varphi(\eta(X,\Lambda))=(\eta_1(X,0),\ldots,\eta_p(X,0))\in \Lift(f)$.

Now let $\overline{\eta}\in \Lift(f)$. By definition, there exists a vector
field $\overline{\xi}\in \theta_n$ such that
$df(\overline{\xi})=\overline{\eta}\circ f$. Define $\eta(X,
\Lambda)=(\overline{\eta}(X), 0)$ and $\xi(x,
\lambda)=(\overline{\xi}(x), 0)$. Consider
$\widetilde{\eta}=\eta\circ F-dF(\xi)\in \theta_{S\times\{0\}}(F)$.
We have $\widetilde{\eta}(x,0)=(0,0)$ so by Hadamard's Lemma there exist vector fields
$\widetilde{\eta}_i\in \theta(F)$ with $1\leq i\leq r$ such that
$\widetilde{\eta}=\sum_{i=1}^r\Lambda_i\widetilde{\eta}_i$.

Since $F$ is stable,  there exist $\hat{\xi}_i\in \theta_{n+r}$ and $\hat{\eta}_i\in \theta_{p+1}$ for $1\leq i\leq r$ such that
$\widetilde{\eta}_i=dF(\hat{\xi}_i)+\hat{\eta}_i\circ F$ for all $i$. Therefore,
$$
\eta\circ F - dF(\xi)=\widetilde{\eta}=\sum_{i=1}^r\Lambda_i\widetilde{\eta}_i=
\sum_{i=1}^r\lambda_i(dF(\hat{\xi}_i)+\hat{\eta}_i\circ F)=
dF(\sum_{i=1}^r\lambda_i\hat{\xi}_i)+(\sum_{i=1}^r\Lambda_i\hat{\eta}_i)\circ F.
$$
Since $\eta$ is 0 in the last $r$ components, we have that $\eta-\sum_{i=1}^r\Lambda_i\hat{\eta}_i\in \Lift(F)\cap G$ and $\varphi(\eta-\sum_{i=1}^r\Lambda_i\hat{\eta}_i)=\overline{\eta}$.
\end{proof}

\begin{rem}
If $r=1$, then $G=\Lift(g)$, where $g:(\mathbb K^p\times\mathbb K,0)\rightarrow(\mathbb K^p\times\mathbb K,0)$ is the map germ $g(X,\Lambda)=(X,\Lambda^2)$ as was stated in Theorem 2 in \cite{NORW}. However, if $r>1$ and $g:(\mathbb K^p\times\mathbb K^r,S)\rightarrow(\mathbb K^p\times\mathbb K^r,0)$, where $S=\{z_1,\dots,z_r\}$ and the branch at $z_i$ is given by $g_i(X,\Lambda)=(X,\Lambda_1,\ldots,\Lambda_i^2,\ldots,\Lambda_r)$, then 
$$\Lift(g)=\mathcal O_{p+r}\left\{\frac{\partial}{\partial X_k},1\leq k\leq p,\Lambda_i\frac{\partial}{\partial \Lambda_i},1\leq i\leq r\right\}\varsubsetneq G.$$
\end{rem}

\subsection{Liftable vector fields of $H_k$}

As an application of Theorem \ref{lift} we compute the liftable vector fields of a family of germs which do not admit a 1-parameter stable unfolding. 
Consider $H_k:(\mathbb K^2,0)\rightarrow (\mathbb K^3,0)$, from Mond's list (\cite{mond}) given by
$$
H_k(x,y)=(x,y^3,y^{3k-1}+xy).
$$ 
For any $k\geq 2$, $H_k$ admits a 2-parameter stable unfolding $\A$-equivalent to $F:(\mathbb K^4,0)\rightarrow (\mathbb K^5,0)$, 
$$
F(u_1,v_1,v_2,y)=(u_1,v_1,v_2,y^3+u_1y,v_1y+v_2y^2),
$$ 
which is versal for $k=2$. In fact, the change of variables in the source $v_2'=v_2-y^{3k-3}$, the change $G_k(U_1,V_1,V_2,W_1,W_2)=(U_1,V_1,V_2-W_1^{k-1},W_1,W_2)$ in the target, and a further change in the source, show that $F$ is $\A$-equivalent to 
$$
F_k(u_1,v_1,v_2,y)=(u_1,v_1,v_2,y^3+u_1y,v_1y+v_2y^2+\psi(u_1,y)y^2+y^{3k-1}),
$$
for some function $\psi$ such that $\psi(0,y)=0$. So $F_k$ is a 2-parameter stable unfolding of $H_k$.

In this case, 
$$G=\mathcal O_{3+2}\left\{U_1\frac{\partial}{\partial U_1},V_2\frac{\partial}{\partial U_1},U_1\frac{\partial}{\partial V_2},V_2\frac{\partial}{\partial V_2},\frac{\partial}{\partial V_1},\frac{\partial}{\partial W_1},\frac{\partial}{\partial W_2}\right\},
$$
and by Theorem \ref{lift},  $\Lift(H_k)=\varphi(\Lift(F_k)\cap G)$, where $\varphi=\pi_1\circ i^*$.

In Example 3.4 in \cite{houstonlittle}, the generators for $\Lift(F)$ are calculated. We show these vector fields as matrices for simplicity:
$$
\eta_e=\left(\begin{array}{c}
2U_1\\
2V_1\\
V_2\\
3W_1\\
3W_2\end{array}\right),
\eta_1^{1,2,3}=\left(\begin{array}{c}
4U_1^2\\
-3U_1V_1+3V_2W_1\\
-5U_1V_2-3W_2\\
6U_1W_1\\
-3V_1W_1+2U_1W_2\end{array}\right),
\left(\begin{array}{c}
6U_1\\
-3V_1\\
-6V_2\\
9W_1\\
0\end{array}\right),
\left(\begin{array}{c}
9V_1\\
-6V_2^2\\
0\\
9W_2+3U_1V_2\\
3V_1V_2\end{array}\right),$$

$$\eta_2^{1,2,3}=\left(\begin{array}{c}
0\\
-3U_1V_2-3W_2\\
3V_1\\
0\\
-3V_2W_1\end{array}\right),
\left(\begin{array}{c}
-9W_1\\
2U_1V_2\\
-3V_1\\
2U_1^2\\
6V_2W_1+2U_1V_1\end{array}\right),
\left(\begin{array}{c}
-9W_2-3U_1V_2\\
-3V_1V_2\\
0\\
3U_1V_1\\
6V_2W_2+3V_1^2\end{array}\right).$$

In \cite[Lemma 6.1]{NORW}, it is shown that if two maps $f$ and $g$ are $\A$-equivalent, i.e. there exist $h$ and $H$ diffeomorphisms in source and target such that $f=H\circ g\circ h$, then the map $L_{f,g}:\Lift(g)\rightarrow \Lift(f)$ such that $L_{f,g}(\eta)=dH\circ\eta\circ H^{-1}$ is a bijection. We therefore can obtain $\Lift(F_k)$ from $\Lift(F)$ by composing its generators by $G_k^{-1}$ on the right and multiplying by $dG_k$ on the left.
$$
\eta_{ke}=\left(\begin{array}{c}
2U_1\\
2V_1\\
V_2-(3k-4)W_1^{k-1}\\
3W_1\\
3W_2\end{array}\right),
\eta_{k1}^1=\left(\begin{array}{c}
4U_1^2\\
-3U_1V_1+3V_2W_1+3W_1^k\\
-5U_1V_2-(6k-1)U_1W_1^{k-1}-3W_2\\
6U_1W_1\\
-3V_1W_1+2U_1W_2\end{array}\right),$$

$$
\eta_{k1}^2=\left(\begin{array}{c}
6U_1\\
-3V_1\\
-6V_2-(9k-3)W_1^{k-1}\\
9W_1\\
0\end{array}\right),
\eta_{k2}^1=\left(\begin{array}{c}
0\\
-3U_1V_2-3U_1W_1^{k-1}-3W_2\\
3V_1\\
0\\
-3V_2W_1-3W_1^k\end{array}\right),$$

$$
\eta_{k1}^3=\left(\begin{array}{c}
9V_1\\
-6V_2^2-12V_2W_1^{k-1}-6W_1^{2k-2}\\
-9(k-1)W_1^{k-2}W_2-3(k-1)U_1V_2W_1^{k-2}-3(k-1)U_1W_1^{2k-3}\\
9W_2+3U_1V_2+3U_1W_1^{k-1}\\
3V_1V_2+3V_1W_1^{k-1}\end{array}\right),$$

$$
\eta_{k2}^{2}=\left(\begin{array}{c}
-9W_1\\
2U_1V_2+2U_1W_1^{k-1}\\
-3V_1-2(k-1)U_1^2W_1^{k-2}\\
2U_1^2\\
6V_2W_1+6W_1^k+2U_1V_1\end{array}\right),
\eta_{k2}^{3}=\left(\begin{array}{c}
-9W_2-3U_1V_2-3U_1W_1^{k-1}\\
-3V_1V_2-3V_1W_1^{k-1}\\
-3(k-1)U_1V_1W_1^{k-2}\\
3U_1V_1\\
6V_2W_2+6W_1^{k-1}W_2+3V_1^2\end{array}\right).$$

The module $\Lift(F_k)\cap G$ is formed by vector fields of type 
$$\eta=\tilde{\alpha}_1\eta_{ke}+\tilde{\alpha}_2\eta_{k1}^1+\tilde{\alpha}_3\eta_{k1}^2+\tilde{\alpha}_4\eta_{k2}^1+\tilde{\alpha}_5\eta_{k1}^3+\tilde{\alpha}_6\eta_{k2}^2+\tilde{\alpha}_7\eta_{k2}^3,$$ such that the first and third components are divisible by $U_1$ or $V_2$, where the $\tilde{\alpha_i}$ are functions in $U_1,V_1,V_2,W_1$ and $W_2$. This means that the first and third components evaluated at $U_1=V_2=0$ are equal to 0, which yields two conditions on $\alpha_i(V_1,W_1,W_2)=\tilde{\alpha}_i(0,V_1,0,W_1,W_2)$:
$$9\alpha_5V_1-9\alpha_6W_1-9\alpha_7W_2=0,$$
$$-(3k-4)\alpha_1W_1^{k-1}-3\alpha_2W_2-(9k-3)\alpha_3W_1^{k-1}+3\alpha_4V_1-9(k-1)\alpha_5W_1^{k-2}W_2-3\alpha_6V_1=0.$$ 
Now, set $\beta=-(3k-4)\alpha_1-(9k-3)\alpha_3$, so that $\alpha_3=-\frac{1}{9k-3}\beta-\frac{3k-4}{9k-3}\alpha_1$. The conditions can be rewritten as 
\begin{align}\label{conditions}
\alpha_5V_1&=\alpha_6W_1+\alpha_7W_2,\\
\alpha_4V_1&=-\frac{1}{3}\beta W_1^{k-1}+\alpha_2W_2+3(k-1)\alpha_5W_1^{k-2}W_2+\alpha_6V_1.
\end{align}
 
By Theorem \ref{lift}, the liftable vector fields of $H_k$ are obtained from the second, fourth and fifth components vector fields of the type of $\eta$ evaluated at $U_1=V_2=0$. That is, 
$$\begin{array}{ll}\Lift(H_k)= & \{\alpha_1(2V_1\frac{\partial}{\partial V_1}+3W_1\frac{\partial}{\partial W_1}+3W_2\frac{\partial}{\partial W_2})+\alpha_2(3W_1^k\frac{\partial}{\partial V_1}-3V_1W_1\frac{\partial}{\partial W_2})\\
&+\alpha_3(-3V_1\frac{\partial}{\partial V_1}+9W_1\frac{\partial}{\partial W_1})+\alpha_4(-3W_2\frac{\partial}{\partial V_1}-3W_1^k\frac{\partial}{\partial W_2})\\
&+\alpha_5(-6W_1^{2k-2}\frac{\partial}{\partial V_1}+9W_2\frac{\partial}{\partial W_1}+3V_1W_1^{k-1}\frac{\partial}{\partial W_2})+\alpha_6(6W_1^k\frac{\partial}{\partial W_2})\\
&+\alpha_7(-3V_1W_1^{k-1}\frac{\partial}{\partial V_1}+(6W_1^{k-1}W_2+3V_1^2)\frac{\partial}{\partial W_2})\},\end{array}$$ where the $\alpha_i$ are germs of functions.

We use conditions (1) and (2) to eliminate the functions $\alpha_4,\alpha_5$. Then, after some easy simplifications in the obtained generators, we arrive to the following:

\begin{teo}
The module of liftable vector fields over the map germs $H_k:(\mathbb K^2,0)\rightarrow (\mathbb K^3,0)$, given by $H_k(x,y)=(x,y^3,y^{3k-1}+xy)$ is generated by the following vector fields
$$
\left(\begin{array}{c}
(3k-2)V_1\\
3W_1\\
(3k-1)W_2\end{array}\right),
\left(\begin{array}{c}
V_1^2+(3k-1)W_1^{k-1}W_2\\
-3V_1W_1\\
(3k-1)W_1^{2k-1}\end{array}\right),
\left(\begin{array}{c}
W_2^2-V_1W_1^k\\
0\\
V_1^2W_1+W_1^kW_2\end{array}\right),$$
$$\left(\begin{array}{c}
(3k-1)W_1^{2k-1}+V_1W_2\\
-3W_1W_2\\
-(3k-1)V_1W_1^k\end{array}\right),
\left(\begin{array}{c}
-(3k-1)W_1^{2k-2}W_2-V_1^2W_1^{k-1}\\
3W_2^2\\
3kV_1W_1^{k-1}W_2+V_1^3\end{array}\right).
$$
\end{teo}

\begin{rem}
In fact, these vector fields can be obtained by applying $\varphi$ to the following vector fields in $\Lift(F_k)\cap G$:
\begin{align*}
&(9k-3)\eta_{ke}-(3k-4)\eta_{k1}^2, \\
&3V_1\eta_{k1}^2+(9k-3)W_1^{k-1}\eta_{k2}^1, \\
&2V_1\eta_{k1}^1+W_2(\eta_{k2}^1-\eta_{k2}^2)+W_1\eta_{k2}^3, \\
&W_1\eta_{k1}^3-3(k-1)W_1^{k-1}\eta_{k1}^1+V_1(\eta_{k2}^1+\eta_{k2}^2),\\
&W_2\eta_{k1}^3+V_1\eta_{k2}^3-3(k-1)W_1^{k-2}W_2\eta_{k2}^1.
\end{align*}
\end{rem}

\section{Liftable vector fields and augmentations}

In this section, we restrict ourselves to the complex case $\K=\C$.
Suppose $f:(\mathbb C^n,S)\longrightarrow (\mathbb C^p,0)$ admits a one-parameter stable unfolding $F(x,\lambda)=(f_{\lambda}(x),\lambda)$. Let $\phi:(\C,0)\rightarrow (\C,0)$, and let $Af(x,z)=(f_{\phi(z)}(x),z)=(X,Z)$ be the augmentation of $f$ by $F$ and $\phi$. We shall establish a relation between the liftable vector fields of an augmentation $Af$ and those of the stable unfolding $F$. 

As we pointed out in Section \ref{section-notation}, if $g$ is a complex analytic germ of finite $\A_e$-codimension, then $\Lift(g)=\Derlog(\Delta(g))$. When $p\ge n+1$, $\Delta(g)$ is always a hypersurface in $(\C^p,0)$, hence $\Derlog(h)\subset \Derlog(\Delta(g))$, where $h$ is the reduced defining equation for $\Delta(g)$ and 
$$\Derlog(h)=\{\eta\in\theta_p:\ \eta(h)=0\}.
$$ 
Furthermore, if $g$ is quasihomogeneous $\Derlog(\Delta(g))=\langle e \rangle\oplus Derlog(h)$, where $e$ is the Euler vector field (see for instance \cite{damon-mond}).

Let $H$ be the defining equation of $\Delta(F)$, then we can take $h(X,Z)=H(X,\phi(Z))$ as the defining equation of $\Delta(Af)$. In particular,
$$
\frac{\partial h}{\partial X_i}(X,Z)=\frac{\partial H}{\partial X_i}(X,\phi(Z)),\ i=1,\dots,p,\quad \frac{\partial h}{\partial Z}(X,Z)=\frac{\partial H}{\partial \Lambda}(X,\phi(Z))\phi'(Z).$$ 
Given a vector field $\eta(X,\Lambda)=\sum_{i=1}^p\eta_i(X,\Lambda)\frac{\partial}{\partial X_i}+\eta_{p+1}(X,\Lambda)\frac{\partial}{\partial \Lambda}\in\theta_{p+1}$, we set:
$$
\widetilde\eta(X,Z)=\sum_{i=1}^p\eta_i(X,\phi(Z))\phi'(Z)\frac{\partial}{\partial X_i}+\eta_{p+1}(X,\phi(Z))\frac{\partial}{\partial Z}\in \theta_{p+1}.
$$
We have the following two lemmas:

\begin{lem}\label{derlogs1}
If $\eta\in \Derlog(\Delta(F))$, then  $\widetilde\eta\in \Derlog(\Delta(Af))$. Furthermore, suppose that $\eta_{p+1}(X,0)=0$, then $\frac{\widetilde\eta}{\phi'}\in \Derlog(\Delta(Af))$. 
\end{lem}

\begin{proof}
If $\eta\in \Derlog(\Delta(F))$, $\eta(H)=\alpha H$ for some $\alpha\in\mathcal O_{p+1}$. That is, 
$$
\sum_{i=1}^p\eta_i(X,\Lambda)\frac{\partial H}{\partial X_i}(X,\Lambda)+\eta_{p+1}(X,\Lambda)\frac{\partial H}{\partial \Lambda}(X,\Lambda)=\alpha(X,\Lambda) H(X,\Lambda).
$$ 
Substituting $Z=\phi(\Lambda)$ and multiplying by $\phi'(Z)$ yields
\begin{align*}
&\sum_{i=1}^p\eta_i(X,\phi(Z))\phi'(Z)\frac{\partial H}{\partial X_i}(X,\phi(Z))+\eta_{p+1}(X,\phi(Z))\phi'(Z)\frac{\partial H}{\partial \Lambda}(X,\phi(Z))\\&=\alpha(X,\phi(Z)) \phi'(Z)H(X,\phi(Z)).
\end{align*}
Hence, $\widetilde\eta(h)=\phi'\widetilde \alpha h$, where $\widetilde \alpha(X,Z)=\alpha(X,\phi(Z))$ and thus 
$\widetilde\eta\in \Derlog(\Delta(Af))$.

On the other hand, if $\eta_{p+1}(X,0)=0$ we can write $\eta_{p+1}=\Lambda a$, for some $a\in \O_{p+1}$.  We also put $\phi= b\phi'$, for some $b\in\O_1$. We have:
$$
\eta_{p+1}(X,\phi(Z))=\phi(Z)a(X,\phi(Z))=b(Z)\phi'(Z)a(X,\phi(Z)).
$$
Thus $\widetilde\eta$ is divisible by $\phi'$ and  $\frac{\widetilde\eta(h)}{\phi'}=\widetilde \alpha h$, so $\frac{\widetilde\eta}{\phi'}\in \Derlog(\Delta(Af))$.
\end{proof}

\begin{lem}\label{derlogs2}
Suppose $\phi$ is quasihomogeneous and let $\overline{\eta}\in \Derlog(\Delta(Af))$.Then there exists $\eta\in \Derlog(\Delta(F))$ such that $\widetilde\eta=\overline{\eta}+\varphi$ with $\varphi\in \Derlog(\Delta (Af))$. Furthermore, if $\overline{\eta}_{p+1}(X,0)= 0$ then there exists $\eta\in \Derlog(\Delta(F))$ such that $\frac{\widetilde\eta}{\phi'}=\overline{\eta}+\varphi$, with $\varphi\in\Derlog(\Delta(Af))$.
\end{lem}

\begin{proof}
We put $\overline\eta_{p+1}(X,Z)=Z\psi(X,Z)+\beta(X)$, for some functions $\psi,\beta$, hence
$$
\overline\eta(X,Z)=\sum_{i=1}^p\overline{\eta}_i(X,Z)\frac{\partial}{\partial X_i}+(Z\psi(X,Z)+\beta(X))\frac{\partial}{\partial Z}.
$$
For simplicity we suppose $\overline\eta\in\Derlog(h)$ and comment on the general case later. So, we have
$$
\sum_{i=1}^{p}\overline{\eta}_i(X,Z)\frac{\partial h}{\partial X_i}(X,Z)=-(Z\psi(X,Z)+\beta(X))\frac{\partial h}{\partial Z}(X,Z),
$$ 
which means 
\begin{align*}
&\sum_{i=1}^{p}\overline{\eta}_i(X,Z)\frac{\partial H}{\partial X_i}(X,\phi(Z))+Z\psi(X,Z)\phi'(Z)\frac{\partial H}{\partial \Lambda}(X,\phi(Z))\\
&=-\beta(X)\frac{\partial H}{\partial \Lambda}(X,\phi(Z))\phi'(Z).
\end{align*}
Since $\phi(Z)=Z^k$ and $\phi'(Z)=kZ^{k-1}$, the right hand side of this equality can be seen as a series in $Z$ of type 
$$
c_{k-1}(X)Z^{k-1}+c_{2k-1}(X)Z^{2k-1}+c_{3k-1}(X)Z^{3k-1}+\ldots
$$
On the other hand, for $i=1,\ldots,p$, 
\begin{align*}
\overline{\eta}_i(X,Z)&=a_{i0}(X)+a_{i1}(X)Z+a_{i2}(X)Z^2+\ldots,\\ 
\frac{\partial H}{\partial X_i}(X,\phi(Z))&=b_{i0}(X)+b_{ik}(X)Z^k+b_{i,2k}(X)Z^{2k}+\ldots
\end{align*}
Analogously for $Z\psi(X,Z)$ and $\frac{\partial H}{\partial \Lambda}(X,\phi(Z))$. By identifying the coefficients at both sides of the equality, the left hand side has a part equal to 0 (the sum of the coefficients whose terms do not appear in the right hand side) and the rest is used to rewrite the equality
\begin{align*}
&\sum_{i=1}^{p}A_{i}(X,\phi(Z))\phi'(Z)\frac{\partial H}{\partial X_i}(X,\phi(Z))+ A_{p+1}(X,\phi(Z))\phi'(Z)\frac{\partial H}{\partial \Lambda}(X,\phi(Z))\\
&=-\beta(X)\phi'(Z)\frac{\partial H}{\partial \Lambda}(X,\phi(Z)),
\end{align*}
for some functions $A_{i}$, $i=1,\dots,p+1$.
The part which was equal to 0 corresponds to a certain $\varphi\in \Derlog(h)$.

Dividing the previous equation by $\phi'(Z)$ and substituting $\Lambda=\phi(Z)$, we obtain 
$$\sum_{i=1}^{p}A_{i}(X,\Lambda)\frac{\partial H}{\partial X_i}(X,\Lambda)+(A_{p+1}(X,\Lambda)+\beta(X))\frac{\partial H}{\partial \Lambda}(X,\Lambda)=0,
$$ which means that 
$$\eta=\sum_{i=1}^{p}A_{i}\frac{\partial}{\partial X_i}+(A_{p+1}+\beta)\frac{\partial}{\partial Z}\in \Derlog(H)
$$
and $\widetilde\eta=\overline\eta-\varphi$.

If $\overline{\eta}\not\in \Derlog(h)$, that is $\overline{\eta}(h)=\alpha(X,Z)h(X,Z)=\alpha(X,Z)H(X,\phi(Z))$, in the same way as above we would take only the part of $\alpha(X,Z)H(X,\phi(Z))$ in the powers $Z^{nk-1}$, that is $\gamma(X,\phi(Z))\phi'(Z)H(X,\phi(Z))$ and proceed in the same way.

Second, suppose that $\overline{\eta}_{p+1}(X,0)=0$, so $\beta=0$. 
Then $\overline{\eta}(h)=\alpha h$ yields 
$$
\sum_{i=1}^{p}\overline{\eta}_i(X,Z)\frac{\partial H}{\partial X_i}(X,\phi(Z))+Z\psi(X,Z)\phi'(Z)\frac{\partial H}{\partial \Lambda}(X,\phi(Z))=\alpha(X,Z)H(X,\phi(Z)).
$$ 
If we take only the terms in $Z^{nk}$ with $n\in\{0,1,\ldots\}$ on both sides of the equation we are left with 
\begin{align*}
&\sum_{i=1}^{p}A_i(X,\phi(Z))\frac{\partial H}{\partial X_i}(X,\phi(Z))+ZA_{p+1}(X,\phi(Z))\phi'(Z)\frac{\partial H}{\partial \Lambda}(X,\phi(Z))\\
&=\gamma(X,\phi(Z))H(X,\phi(Z)).
\end{align*} 
Notice that $Z\phi'(Z)=kZ^k=k\phi(Z)$, so if we substitute $\Lambda=\phi(Z)$, we get that 
$$
\eta=\sum_{i=1}^{p}A_i\frac{\partial}{\partial X_i}+k\Lambda A_{p+1}\frac{\partial}{\partial \Lambda}\in \Derlog(\Delta(F)).
$$
\end{proof}

\begin{rem}
Lemmas \ref{derlogs1} and \ref{derlogs2} imply that the number of generators of $\Derlog(\Delta(Af))$ (and therefore of $\Lift(Af)$) is greater than or equal to the number of generators of $\Derlog(\Delta(F))$ ($\Lift(F)$ resp.). Notice too that if $\eta\in \Derlog(H)$, then $\widetilde\eta\in\Derlog(h)$.
\end{rem}


\begin{ex}
Consider $f:(\mathbb C^2,0)\rightarrow (\mathbb C^2,0)$ given by $f(x,y)=(x^4+yx,y)$, its 1-parameter stable unfolding $F(x,y,z)=(x^4+yx+zx^2,y,z)$ and the augmentations by $\phi(z)=z^k$, $A^kf(x,y,z)=(x^4+yx+z^kx^2,y,z)$. The defining equation of the discriminant is 
$$
h(X,Y,Z)=256X^3+27Y^4+144XY^2Z^k+128X^2Z^{2k}+4Y^2Z^{3k}+16XZ^{4k}.
$$ 
When $k=1$ we have $H$. It can be seen (see \cite{arnold} for example) that 
$\Lift(F)=\langle \eta^1,\eta^2,\eta^3 \rangle$, where
\begin{align*}
\eta^1&=4X\frac{\partial}{\partial
X}+3Y\frac{\partial}{\partial Y}+2Z\frac{\partial}{\partial
Z},\\
\eta^2&=-(9Y^2+16XZ)\frac{\partial}{\partial
X}+12YZ\frac{\partial}{\partial
Y}+(48X+4Z^2)\frac{\partial}{\partial Z},\\
\eta^3&=YZ\frac{\partial}{\partial
X}-(8X+2Z^2)\frac{\partial}{\partial
Y}+6Y\frac{\partial}{\partial Z}\rangle.
\end{align*}
On the other hand, $\Lift(A^kf)=\langle \widetilde{\eta}^1,\widetilde{\eta}^2,\widetilde{\eta}^3,\widetilde{\eta^4} \rangle$, where
\begin{align*}
\widetilde\eta^1&=4kX\frac{\partial}{\partial
X}+3kY\frac{\partial}{\partial Y}+2Z\frac{\partial}{\partial
Z},\\
\widetilde \eta^2&=-(9kY^2Z^{k-1}+16kXZ^{2k-1})\frac{\partial}{\partial
X}+12kYZ^{2k-1}\frac{\partial}{\partial
Y}+(48X+4Z^{2k})\frac{\partial}{\partial Z},\\
\widetilde \eta^3&=kYZ^{2k-1}\frac{\partial}{\partial
X}-(8kXZ^{k-1}+2kZ^{3k-1})\frac{\partial}{\partial
Y}+6Y\frac{\partial}{\partial Z},\\
\widetilde \eta^4&=-(9kY^3+24kXYZ^{k})\frac{\partial}{\partial
X}+(64kX^2+12kY^2Z^k+16kXZ^{2k})\frac{\partial}{\partial
Y}\\
&\quad+4YZ^{k+1}\frac{\partial}{\partial Z}.
\end{align*}

Following Lemma \ref{derlogs1}, if we substitute $Z$ by $Z^k$ in $\eta^2$ and $\eta^3$ and multiply their first two components by $kZ^{k-1}$ we obtain $\widetilde{\eta}^2$ and $\widetilde{\eta}^3$. Now consider $\xi=kY\eta^2(X,Y,Z)-8kX\eta^3(X,Y,Z)$. Since $\xi_3(X,Y,0)=0$, again by Lemma \ref{derlogs1}, substituting $Z$ by $Z^k$ and then dividing the last component by $\phi'(Z)=kZ^{k-1}$ gives $\widetilde{\eta}^4$. Finally, if we take $k\eta^1(X,Y,Z)$, substitute $Z$ by $Z^k$ and then divide the last component by $\phi'(Z)=kZ^{k-1}$ we get $\widetilde{\eta}^1$.
\end{ex}

In \cite{ORW1} the operation of simultaneous augmentation and concatenation was defined as the multigerm defined by $\{Af,g\}$ where $g$ is a fold map or an immersion. The authors proved a formula relating the $\A_e$-codimensions of $\{Af,g\}$ and of $f$ under a certain condition, namely that $\pi_2(i^*(\Lift(Af)))=\pi_2(i^*(\Lift(F)))$ where $\pi_2$ is the projection to the last component and $i$ is the immersion $i(X)=(X,0)$. This condition appears again in \cite{casonattooset} and in \cite{ORW2} where the authors say that all examples they have studied satisfy this condition but they were not able to prove it. We will use Theorem \ref{lift} and Lemmas \ref{derlogs1}, \ref{derlogs2} in order to give a proof of this fact when the augmenting function $\phi$ is quasihomogeneous.

\begin{teo}
Let $f:(\mathbb C^n,S)\longrightarrow (\mathbb C^p,0)$ be a map germ with $p\le n+1$ which admits a one-parameter stable unfolding $F$. Let $\phi:(\mathbb C,0)\rightarrow (\mathbb C,0)$ be a quasihomogeneous function and let $Af$ be the augmentation of $f$ by $F$. Then, $$\pi_2(i^*(\Lift(Af)))=\pi_2(i^*(\Lift(F))).$$
\end{teo}
\begin{proof}
Let us see that $\pi_2(i^*(\Lift(F))\subseteq\pi_2(i^*(\Lift(Af)))$. Suppose $F(x,\lambda)=(f_\lambda(x),\lambda)$ and consider 
$$AF(x,z,\mu)=(f_{\phi(z)+\mu}(x),z,\mu),
$$ 
the one-parameter stable unfolding of $Af$ and the immersion $i_1(X,Z)=(X,Z,0)$. Then, by Theorem \ref{lift}, 
$$
\Lift(Af)=\varphi(\Lift(AF)\cap G)=\pi_1(i_1^*(\Lift(AF)\cap G)),
$$ where $G=\mathcal O_{p+1+1}\{\frac{\partial}{\partial X_k},\frac{\partial}{\partial Z},M\frac{\partial}{\partial M}\}$ and $\pi_1$ is the projection on the first $p+1$ components. On the other hand, $AF$ is $\mathcal A$-equivalent to a trivial one-parameter unfolding of $F$, and if we consider the immersion $i_1'(X,M)=(X,0,M)$, then, by Theorem \ref{lift}, 
$$\Lift(F)=\varphi'(\Lift(AF)\cap G')=\pi'_1({i_1'}^*(\Lift(AF)\cap G')),$$ where $G'=\mathcal O_{p+1+1}\{\frac{\partial}{\partial X_k},Z\frac{\partial}{\partial Z},\frac{\partial}{\partial M}\}$ and $\pi'_1$ is the projection onto the components in $\frac{\partial}{\partial X_k}$ and $\frac{\partial}{\partial M}$.

Now let $\chi=\beta(X)\frac{\partial}{\partial M}\in\pi_2(i^*(\Lift(F))=\pi_2(i^*(\pi'_{1}({i_1'}^*(\Lift(AF)\cap G')))$. There exists $\xi\in \Lift(AF)\cap G'$ such that $$\xi=\sum_{i=1}^p\xi_i\frac{\partial}{\partial X_i}+\xi_{p+1}\frac{\partial}{\partial Z}+(\beta+\xi_{p+2})\frac{\partial}{\partial M},$$ where $\xi_{p+1}(X,0,M)=0$ and $\xi_{p+2}(X,0,0)=0$. Notice that $\dim\widetilde\tau(AF)>0$, in fact $\frac{\partial}{\partial Z}\in\widetilde\tau(AF)$ so there exists $\eta\in \Lift(AF)$ such that $$\eta=\sum_{i=1}^p\eta_i\frac{\partial}{\partial X_i}+\eta_{p+1}\frac{\partial}{\partial Z}+\eta_{p+2}\frac{\partial}{\partial M},$$ where $\eta_{p+1}(0)=1$.

We define $\zeta=\frac{1}{\eta_{p+1}}(\eta_{p+2}\xi-(\beta+\xi_{p+2})\eta)$ and we get $$\zeta=\frac{1}{\eta_{p+1}}(\sum_{i=1}^p(\eta_{p+2}\xi_i-(\beta+\xi_{p+2})\eta_i)\frac{\partial}{\partial X_i}+(\xi_{p+1}\eta_{p+2}-(\beta+\xi_{p+2})\eta_{p+1})\frac{\partial}{\partial Z}).$$ We have that $\zeta\in \Lift(AF)\cap G$, so $\chi\in\pi_2(i^*(\pi_1'(i_1^*(\Lift(AF)\cap G)))=\pi_2(i^*(\Lift(Af)))$.

\bigskip

Now we will prove $\pi_2(i^*(\Lift(Af))\subseteq\pi_2(i^*(\Lift(F)))$. 
Let $\chi=\beta(X)\frac{\partial}{\partial M}\in\pi_2(i^*(\Lift(Af))$, i.e. there exists $\overline{\eta}\in \Lift(Af)=\Derlog(\Delta(Af))$ such that $$\overline{\eta}(X,Z)=\sum_{i=1}^p\overline{\eta}_i(X,Z)\frac{\partial}{\partial X_i}+(\psi_{p+1}(X,Z)+\beta(X))\frac{\partial}{\partial Z},$$
where $\psi_{p+1}(X,0)=0$. By Lemma \ref{derlogs2}, there exists $\eta\in \Derlog(\Delta(F))$ of type $$\eta=\sum_{i=1}^{p}A_{i,k-1}\frac{\partial}{\partial X_i}+(\eta_{p+1}+\beta)\frac{\partial}{\partial Z},$$ where $\eta_{p+1}(X,0)=\psi_{p+1}(X,0)=0$, and so $\pi_2(i^*(\eta))=\chi$. The result is proved.

\end{proof}

\end{document}